\documentclass[letterpaper, 10 pt, conference]{ieeeconf}  

\IEEEoverridecommandlockouts                              
\title{\LARGE \bf
Singular Arcs on Average Optimal Control-Affine Problems 
}

\usepackage{graphicx}
\usepackage{amssymb,amsmath,amsthm,mathtools}
\usepackage{color}
\usepackage{cancel}
\usepackage{changes}
\usepackage{enumerate}
\usepackage{float}

\usepackage[export]{adjustbox}
\usepackage{subcaption}
\usepackage{svg}

\newtheorem{theorem}{Theorem}[section]

\newtheorem{lemma}[theorem]{Lemma}

\theoremstyle{definition}
\newtheorem{definition}{Definition}[section]

\def\O{\Omega}
\def\o{\omega}

\DeclareMathOperator\supp{supp}
\DeclareMathOperator*{\argmax}{arg\,max}

\author{M.S. Aronna$^{1}$
, G. de Lima Monteiro$^{1}$ and O. Sierra$^{1}$ 
\thanks{The authors acknowledge the support of the Brazilian agencies FAPERJ (Rio de Janeiro)  for processes E-28/201.346/2021, 210.037/2024, SEI 260003/000175/2024
and  CAPES 88887.488134/2020-00 and CNPq 312407/2023-8}
\thanks{ $^{1}$ M.S. Aronna, G. de Lima Monteiro and O. Sierra are with the Escola de Matem\'atica Aplicada, Funda\c c\~ ao Getulio Vargas EMAp-FGV, Praia de Botafogo 190,   22250-900 Rio de Janeiro - RJ, Brazil.} \thanks{{\tt\small soledad.aronna@fgv.br}}%
\thanks{{\tt\small g.delimamonteiro@gmail.com}}%
\thanks{{\tt\small oscar.fonseca@fgv.br}}
}
\begin{document}

\maketitle



\begin{abstract}
In this paper we address optimal control problems in which the system parameters follow a probability distribution, and the optimization is based on average performance. These problems, known as {\em Riemann-Stieltjes optimal control} or {\em optimal ensemble control} problems, involve uncertainties that influence system dynamics. Focusing on control-affine systems, we apply the {\em Pontryagin Maximum Principle} to characterize singular arcs in a feedback form. To demonstrate the practical relevance of our approach, we apply it to the {\em sterile insect technique}, a biological pest control method. Numerical simulations confirm the effectiveness of our framework in addressing control problems under uncertainty.
\vspace{0.2cm}

\end{abstract}



\section{Introduction}
This article examines optimal control problems where parameters are described by a probability distribution. Our aim is to optimize the control strategy based on average performance. This category of problems, characterized by internal randomness within the system, is referred to as {\em Riemann-Stieltjes} optimal control problems, {\em average} optimal control, or {\em optimal ensemble} control problems. Our research focuses on the necessary conditions and the characterization of controls in feedback form for a control-affine problem. 
Control-affine systems, characterized by control inputs appearing linearly in their equations of motion, are fundamental in control theory and are described by $\dot{x} = A(x) + B(x) u,$ where $x$ is the state vector, $u$ is the control input, $A(x)$ represents the system's drift dynamics, and $B(x)$ is the control influence matrix. Control-affine models have wide-ranging applications, including robotics and motion planning for trajectory optimization in drones and robotic manipulators \cite{murray1994}, aerospace and autonomous vehicles for fuel-efficient trajectory and attitude control \cite{stengel2004flight}, biological and medical systems for neural stimulation  
among others. Despite their versatility, solving optimal control problems for these systems remains challenging due to nonlinearities and computational demands.

A significant difficulty arises due to the presence of uncertain parameters, represented by $\omega \in \Omega$, which introduce variability and unpredictability into the system dynamics. These uncertainties can stem from modeling inaccuracies, external disturbances, environmental variability, or inherent stochasticity.

In this paper, we deal with the following problem, referred to as problem  $(P)$: 
\begin{equation*}
   \min_{u \in \mathcal{U}_{\rm ad}}J[u(\cdot)], \ \text{where }  J[u(\cdot)]:=\int_\Omega g(x(T,\omega),\omega) \, d\mu(\omega),
\end{equation*}
 subject to the constraints 
 {\small \begin{align*}
\begin{cases}
\dot{x}(t,\omega) = f_0( x(t,\omega),\omega) + \displaystyle f_1(x(t,\omega),\omega)u(t),\\ \hspace{4cm}\text{a.e. } t\in [t_0,T], \ \omega\in \Omega, \\
x(t_0,\omega)=x_0,  \\
u(t) \in U(t), \ \  \text{a.e. } t\in [t_0,T],
\end{cases}
\end{align*}}

\noindent where $f_i\colon \mathbb{R}^n\times\Omega\rightarrow\mathbb{R}^n$ for $i=0,1$ $g 
\colon \mathbb R^n \times \O \rightarrow
\mathbb R$ and  {\small $u:{[t_0,T]}\rightarrow \mathbb{R}$} is the control function, which belongs to the {\em admissible controls set}
{\small \begin{equation}\label{Uad}
\mathcal{U}_{\rm ad} := \left\{ u \in  L^\infty(t_0,T;\mathbb{R}): u(t)\in U(t) \ \text{a.e. } t\in [t_0,T]\right\},
\end{equation}}
where  $U :[0,T] \rightsquigarrow \mathbb R$ is a multifunction taking non-empty closed and convex values contained in a compact subset ${\bf U}\subset \mathbb{R}.$  

The Riemann-Stieltjes framework for uncertain dynamic systems, as discussed in \cite{Ross2015}, has been applied across various disciplines. These include aerospace engineering, as highlighted in \cite{GonzlezArribas2018,Shaffer2016}, military operations \cite{Walton2018}, and affine dynamics for reinforcement learning and system identification \cite{PesareMAY2021,pesare2021convergence}. Numerical investigations were conducted in \cite{Lambrianides2020,Phelps2016}, while \cite{Bettiol2019} extended the theoretical framework to address uncertainties in both dynamics and cost, along with fixed initial conditions. Furthermore, they formulated versions of the Pontryagin Maximum Principle suitable for both smooth and non-smooth scenarios.

In this work, we investigate the existence of optimal controls for problem $(P)$ and apply the Pontryagin Maximum Principle derived in \cite{Bettiol2019} to characterize singular arcs for scalar optimal controls in a feedback form. Finally, we present a model for the implementation of the {\em Sterile Insect Technique} (SIT) as an example. Numerical implementation demonstrates the effectiveness and accuracy of the proposed approach.

The present work is organized as follows. In Section \ref{Preliminaries}, we introduce the preliminaries and notations used throughout the paper. Section \ref{KIC} discusses the existence of optimal controls. In Section \ref{N.O.C}, we verify the hypotheses for the smooth case of the Pontryagin Maximum Principle derived in \cite{Bettiol2019}. Section \ref{control.char} focuses on the characterization of singular arcs in the case of scalar control with commutative vector fields. Section \ref{NUM_EXMP} is dedicated to the numerical implementation of our results, validating the theoretical findings and demonstrating their practical relevance. Finally, the section \ref{conclusions} provides the conclusions.

\section{Preliminaries}\label{Preliminaries}
We use $\mathcal{B}^n$ to represent the collection of Borel subsets of $\mathbb{R}^n$. The notation $\mathcal{B}$- and $\mathcal{L}$- is employed to indicate Borel- and Lebesgue-measurable functions, respectively. The support of a measure $\mu$ defined on $\Omega$ is denoted by $\supp(\mu)$.  

If $X$ is a Banach space, we define $\mathbb{B}$ as the open unit ball centered at the origin in $X$. The space $L^p(0,T;X)$ represents the Lebesgue space of functions defined on the interval $[0,T] \subset \mathbb{R}$ and taking values in $X$. Similarly, $W^{q,s}([0,T];X)$ denotes the Sobolev spaces (see, for instance, Adams \cite{adams1975}). In the special case where $q = r = 1$, we define the corresponding norm as $\|x\|_{W^{1,1}} := \|x\|_{1} + \|\dot{x}\|_{1}$.  

According to \cite{evans2010}, if $ x \in W^{1, p}([0, T]; X) $, then there exists a continuous function $ y \in C([0, T]; X) $ that represents $ x $, meaning that $ x(t) = y(t) $ for almost every $ t \in [s, T] $. The graph of a function $f$ is denoted by ${\rm Gr } f$. A function $\theta:[0,\infty)\rightarrow[0,\infty)$ that is increasing and satisfies $\lim_{s\rightarrow 0^+} \theta(s)=0$ is referred to as a {\em modulus of continuity}.  

Given two smooth vector fields $ f, g: \mathbb{R}^n \rightarrow \mathbb{R}^n $, the {\em Lie bracket} $[f, g]$ defines a new vector field given by:
$$
[f, g] := g'f - f'g,
$$
where, for simplicity, we use the notation $f'$ and $g'$ to represent the Jacobian matrices of $f$ and $g$, respectively.

We introduce the following set of assumptions, which we will refer to as {\bf (H0)} throughout the text:  
\begin{itemize}  
    \item[(i)] The measure $\mu$ is a probability measure defined on a complete separable metric space $(\O, \rho_\O)$. Moreover, for almost every $t \in [t_0, T]$, the set $U(t)$ is a non-empty, compact, and convex subset of $\mathbb{R}^m$, and its graph, $Gr U(\cdot)$, is a $\mathcal{L} \times \mathcal{B}^m$-measurable set.  

    \item[(ii)] For each $i = 0, 1,$ there exist positive constants $c_i$ and $k_i$ such that, for any $x, x' \in \mathbb{R}^n$ and any $\omega \in \Omega$, the function $f_i$ satisfies the following bounds:  
    \begin{equation*}
    |f_i(x,\omega)| \leq c_i  (1+|x|),
    \end{equation*}  
    and the Lipschitz continuity condition:  
    \begin{equation*}
    |f_i(x,\omega)-f_i(x',\omega)|\leq k_i|x-x'|.
    \end{equation*} 
    \item[(iii)] The function $g:\mathbb{R}^n\times\O\rightarrow \mathbb{R}$ is measurable with respect to $\mathcal{B}^n\times\mathcal{B}_\O$. Furthermore, there exist constants $k_g \geq 1$ and $M > 0$ such that, for all $x, x' \in \mathbb{R}^n$ and $\o \in \O$, the function $g$ satisfies:  
    \begin{equation*}
    \begin{aligned}
     |g(x,\o)|&\leq M, \\
     |g(x,\o)-g(x',\o)|& \leq k_g|x-x'|.
    \end{aligned}
    \end{equation*}
\end{itemize}  

An \textit{admissible process} $(u, \{x(\cdot,\o) : \o\in\O\})$ consists of a control function $u \in \mathcal{U}_{\rm ad}$, together with its associated family of trajectories $\{x(\cdot,\o)\in W^{1,1}([0,T], \mathbb{R}^n) : \o\in\O\}$, where for each $\o \in \O$, the function $x(\cdot,\omega)$ satisfies the integral evolution equation:  
{\small \begin{equation}\label{trajectory}
\begin{aligned}
x(t,\o)& = x_0  \\
& \ \ \ \ +\int_{t_0}^{t} \big[ f_0(x(\sigma,\o),\o) + f_1(x(\sigma,\o),\o) u(\sigma) \big] d\sigma.
\end{aligned}
\end{equation}  }

If there exists a process $(\bar{u}, \{\bar{x}(\cdot,\o) : \o\in\O\})$ that solves problem $(P)$, then it is called an {\em optimal pair}. In this case, we refer to $\bar{x}$ as the {\em optimal trajectory} and to $\bar{u}$ as the {\em optimal control}.  

\section{Existence of optimal control}\label{KIC}
To obtain results on the existence of optimal controls and necessary conditions for optimality in problem $(P)$, it is often useful to introduce a general formulation for the system dynamics. Therefore, in suitable cases, we define the function $f$ as  $ f(x,u,\omega):= f_0(x,\omega) + f_1(x,\omega)u.$
We observe that, due to the compactness and convexity of $U(t)$ for almost every $t \in [t_0, T]$, along with the affine structure of the dynamics, for all  $x \in \mathbb{R}^n,$ $\omega \in \Omega.$  
\begin{equation}\label{Comp.Conv}
\begin{aligned}
    f(x, U(t), \omega) \ \text{is compact and convex a.e.} \ t \in [t_0, T].
    \end{aligned}
\end{equation}
Before presenting the existence result, we first study the regularity of the trajectories with respect to the parameters and controls. In the next result, for each control $u \in \mathcal{U}_{\rm ad}, $ the associated trajectory will be denoted by $x_u(\cdot,\omega)$ for all $\omega\in \Omega.$
\begin{lemma}\label{cont.traj}
  Assume that {\bf (H0)} is satisfied. Then, the following properties hold:  
\begin{itemize}
    \item[(i)] For every $u \in \mathcal{U}_{\rm ad}$, the corresponding trajectory $x_u$ remains uniformly bounded:  
    $$|x_u(t,\o)|\leq C_x \quad \text{for all } t\in [t_0,T], \ \o \in \O,$$  
    where $C_x$ is a positive constant that depends on $u$ and the initial condition $x_0$.  

    \item[(ii)] Suppose there exists a modulus of continuity $\theta_{f}$ such that, for any $\o_1,\o_2\in\O,$  
    {\small\begin{equation}\label{f.regular}
    \begin{aligned}
        &\int_{t_0}^T \sup_{x\in\mathbb{R}^n, u\in U(t)} |f(x,u,\o_1) - f(x,u,\o_2)| dt \\ 
        &\hspace{4cm}\leq \theta_{f}(\rho_\O(\o_1,\o_2)),
    \end{aligned}
    \end{equation}}  
    then, for every $u \in \mathcal{U}_{\rm ad}$, the mapping $\omega\mapsto x_u(\cdot,\o)$ is continuous from $\O$ to $W^{1,1}([t_0,T];\mathbb R^n)$.  

    \item [(iii)] For each $\o\in \O,$ the mapping $u\mapsto x_u(\cdot,\o)$ is continuous from ${\mathcal{U}}_{\rm ad}\cap L^2(t_0,T;\mathbb{R}^m)$ to $W^{1,1}([t_0,T];\mathbb R^n)$.
\end{itemize}

\end{lemma}
\begin{proof} 
In what follows, we will define $c:=k_0+\|u\|_{L^{\infty}}k_1.$

 {\it (i)} For all $u \in \mathcal{U}_{\rm ad}$ and $\o\in\O,$  the corresponding trajectory satisfies the equation \eqref{trajectory}, then for all $t\in [t_0,T]$
\begin{align*}
\big|x_{u}(t,\o)&\big| \leq  \big| x_0\big| \\ &+ \int_{t_0}^t \big|f_0( x_{u}(\sigma,\o),\o) + f_1(x_{u}(\sigma,\o),\o)u(\sigma)\big| d\sigma,
\end{align*}
which in turn, in view of ({\bf H0}) and 
 by Grönwall's inequality, we get  $
\big|x_{u}(t,\o)\big|
\leq \left[\big| x_0\big|+ c(t-t_0)\right]e^{c(t-t_0)}
$
and {\it (1)} follows.

 {\it(ii)} Fix a control $u \in \mathcal{U}_{\rm ad}.$      From \eqref{trajectory}, conditions ({\bf H0}) and \eqref{f.regular}  we have that for all $\o_{1},\o_{2}\in \O$ and all $t\in [t_0,T],$
\begin{equation*}
         \begin{aligned}
             &|x_u(t,\o_{1})-x_u(t,\o_{2})|\leq 
(1+\|u\|_{L^{\infty}})\theta_f(\rho_\O(\o_1,\o_2)) \\
&\quad+c\int_{t_0}^{t}|x_u(\sigma,\o_{1})-x_u(\sigma,\o_{2})|\,d\sigma.
         \end{aligned}
     \end{equation*} 
 From Grönwall's inequality, we have that 
 \begin{equation*}
 \begin{aligned}
     &|x_u(t,\o_{1})-x_u(t,\o_{2})|
     \\
     &\ \ \leq
(1+\|u\|_{L^{\infty}})e^{c(t-t_0) }\theta_f(\rho_\O(\o_1,\o_2)),
\end{aligned}
 \end{equation*}
 and the continuity of the map $\o \mapsto x_u(\cdot,\omega)$ follows.
 
 {\it (iii)} Fix $\o\in\O$ and consider two different controls $u,v\in{\mathcal{U}}_{\rm ad}\cap L^2(t_0,T;\mathbb{R}^m)$ with their respective trajectories $x_{u}
(\cdot,\o)$ and $x_{v}(\cdot,\o).$ From \eqref{trajectory}, 
({\bf H0}), {\it (i)} and  Grönwall's inequality we have that
  {\small  \begin{equation*}
        \begin{aligned}
 |x_{u}(t,\o)-x_{v}(t,\o)| 
 \leq C(t-t_0)\left(1+C_x\right)e^{c(t-t_0)}\|u-v\|_{L^{\infty}}.
        \end{aligned}
    \end{equation*}}
 Analogously, we have that 
   {\small \begin{equation*}
        \begin{aligned}
 |x_{u}(t,\o)-x_{v}(t,\o)| 
 \leq C(t-t_0)^{1/2}\left(1+C_x\right)e^{c(t-t_0)}\|u-v\|_{L^{2}}.
        \end{aligned}
    \end{equation*}}
    and the proof ends.
\end{proof}
The proof of the next existence result will crucially depend on the property \eqref{Comp.Conv} and the sequential compactness of trajectories, as stated in \cite[Theorem 23.2]{Clarke}.
\begin{theorem}[Existence of an Optimal Control]\label{Existence}  
Assume that the set of conditions {\bf (H0)} is satisfied. Then, if there exists at least one admissible process for problem $(P)$, the problem admits an optimal global solution.  
\end{theorem}  
\begin{proof}  
By assumption, there exists at least one admissible process $(u, \{ x(\cdot,\omega) : \omega\in\Omega\})$. Furthermore, from condition {\bf (H0)}-(iii), the function $g$ is bounded from below, ensuring that the infimum of the optimal control problem $(P)$ is finite. Consequently, there exists a minimizing sequence, which means that there is a sequence of admissible processes $(u_k, \{x_k(\cdot,\omega) : \omega\in\Omega\})_k$ such that the corresponding cost sequence converges to the infimum of $(P)$.  

By the definition of $\mathcal{U}_{\rm ad}$ in \eqref{Uad}, it follows that the sequence $\{u_k\}_k$ is uniformly bounded. As a result, there exists a subsequence (which we denote using the same indices for simplicity) such that  $
u_k \overset{*}{\rightharpoonup} \bar{u} \quad \text{in} \quad L^\infty(0, T; \mathbb{R}^m),
$
where $\bar{u} \in \mathcal{U}_{\rm ad}.$ For this subsequence, we examine the corresponding sequence of trajectories $\{x_k(\cdot,\omega)\}_k$. By Lemma \ref{cont.traj}-(i), for each $\omega \in \Omega$, the sequence $\{x_k(\cdot,\omega)\}_k$ is uniformly bounded. Furthermore, since the initial condition set $\{x_k(t_0,\omega) : k \in \mathbb{N}\} = \{x_0\}$ is bounded in $\mathbb{R}^n$, and given that \eqref{Comp.Conv} holds, we can apply \cite[Theorem 23.2]{Clarke}. Thus, for each $\omega \in \Omega$, the sequence $\{x_k(\cdot,\omega)\}_k$ has a subsequence (again, unlabeled for simplicity) that converges uniformly to some function $\bar{x}(\cdot, \omega) \in W^{1,1}$. The final step is to verify that the family $\{\bar{x}(\cdot, \omega) : \omega \in \Omega\}$ represents the trajectory corresponding to the control $\bar{u}$.  
 
Let us observe that, from \eqref{trajectory} and condition ({\bf H0})-(ii), it  follows that
\begin{equation*}
\begin{aligned}
    & \left| x_0 + \int_{t_0}^t f_0(x_k(s,\o),\o) +  f_1(x_k(s,\o),\o)u_{k}(s)ds \right. \\
      & \quad \ \ \left. - x_0 - \int_{t_0}^t f_0(\bar x(s,\o),\o) -  f_1(\bar x(s,\o),\o)\bar u(s) ds \right|  \\ &\leq 
  \int_{t_0}^t k_0 \left|x_k(s,\o) -  \bar{x}(s,\o)\right| ds \\ 
 & \ \ \ \ \ +\int_{t_0}^t \left| \left(f_1(x_k(s,\o),\o) - f_1(\bar x(s,\o),\o) \right)\bar u(s)\right| ds \\
   & \ \ \ \ \ + \int_{t_0}^t \left|\left(u_{k}(s) - \bar u(s) \right)f_1(x_k(s,\o),\o) \right| ds.
\end{aligned}
\end{equation*}
Note that, for each $\o\in \O,$ $f_1(x_k(s,\o),\o)$ converges uniformly to $f_1(\bar x(s,\o),\o)$ on $[t_0,T].$ In particular, it converges in $L^1(t_0,T;\mathbb{R}^n).$ Using this fact, and since $u_{k}$ converges weak-* to $\bar{u}$ in $L^{\infty}(t_0,T;\mathbb{R}^n),$ we get that the right-hand side of latter inequality converges to zero.
We get then that 
$$
\bar x(t,\o)= x_0 + \int_{t_0}^t [f_0(\bar x(s,\o),\o) +  f_1(\bar x(s,\o),\o)\bar{u}_{i}(s)]\,ds
$$
which yields that $\bar x(\cdot, \omega)$ is the trajectory associated to $\bar u,$ for each $\omega.$ The remainder of the proof is devoted to showing that the cost function of the sequence converges to the cost of the process $(\bar{u}, \{ \bar{x}(.,\o) : \o\in\O\})$. Recall that $g$ is  uniformly bounded in $\Omega$ in view of condition {\bf(H0)}-(iii). Then, by the Dominated Convergence Theorem,
{\small $$
    \lim_{k \rightarrow \infty} \int_\O g(x_{k}(T,\o),\o) d\mu(\o) = \int_\O g(\bar x(T,\o),\o) d\mu(\o).
$$}
This concludes the proof.
\end{proof}
\section{Necessary optimality conditions}\label{N.O.C}
This section focuses on verifying that the hypotheses of the smooth case version of the Pontryagin maximum principle, as discussed in Bettiol-Khalil \cite{Bettiol2019}, hold within our framework. We start by presenting some key definitions and supplementary regularity assumptions.

\begin{definition}[$W^{1,1}$-local minimizer]
    \label{w11_minimizer}
    A process  $(\bar{u}, \{\bar{x}(\cdot,\o): \o\in \O\})$ is called a {\em $W^{1,1}$-local minimizer} for $(P)$ if there exists $\epsilon>0$ such that
    \begin{equation*}
        \int_\O g(\bar{x}(T,\o),\o)d\mu(\o) \leq
        \int_\O g(x(T,\o),\o)d\mu(\o),
    \end{equation*}
    for every admissible process $(u, \{x(\cdot,\o): \o\in \O\})$ satisfying
    \begin{equation*}
        \| \bar{x} (.,\o) - x(.,\o) \|_{W^{1,1}([0,T];\mathbb{R}^n)} \leq \epsilon \; , \; \forall \o \in \supp (\mu).
    \end{equation*}
\end{definition}

\bigskip

Next, we recall the assumptions and results from Bettiol-Khalil in \cite{Bettiol2019}. Given a $W^{1,1}$-local minimizer $(\bar{u}, \{\bar{x}(\cdot, \omega): \omega \in \Omega\})$ and $\delta > 0$, we refer to the following set of conditions as {\bf (H1)}:

\begin{itemize}\label{A5}
   \item[(i)] There exists a modulus of continuity $\theta_f:[0,\infty)\rightarrow[0,\infty)$ such that for all $\o,\o_1,\o_2\in\O,$ 
   {\small \begin{equation*}
   \begin{aligned}       
   \int_{t_0}^T \sup_{x\in\bar{x}(t,\o)+\delta \mathbb{B}, u\in U(t)} |f(t,x,u,\o_1) &- f(t,x,u,\o_2)| dt \\ 
   &\leq \theta_f(\rho_\O(\o_1,\o_2))
    \end{aligned}
    \end{equation*} }
\item[(ii)] $g(.,\o)$ is differentiable on $\Bar{x}(T,\o)+\delta \mathbb{B}$ for each $\o\in\O,$ with $\nabla_x g(.,\o)$ continuous on $\Bar{x}(T,\o)+\delta \mathbb{B}$ and $\nabla_x g(x,\cdot)$ continuous on $\O$;
\end{itemize}

\begin{itemize}\label{f.contdiff} 
   \item[(iii)] The map $x \mapsto f(t,x,u,\o)$ is continuously differentiable on $\Bar{x}(t,\o)+\delta \mathbb{B}$ for all $u\in U(t), \o\in\O$ almost everywhere in $t \in [t_0,T]$, and $\o \mapsto \nabla_x f(t,x,u,\o)$ is uniformly continuous with respect to $(t,x,u)$.
   \item[(iv)] We assume that the functions $f_i,$ $i=0,1$ are twice continuously differentiable on $\bar{x}(t,\o)+\delta \mathbb{B}$ for all $\o \in \O$ and almost every $t \in [t_0,T]$. Additionally, there exists a constant $C_f^k > 0$ such that for $i \in \{0,1\}$ and $k = 1, 2$, the following conditions hold:
   {\small \begin{equation*}
        \left|\frac{\partial f_i}{\partial x}(\bar{x}(t,\o),\o)\right| < C_{f_i}^1 ,\ \o \in \supp(\mu), \text{ a.e. } t \in [0,T];
    \end{equation*}}
    and
    {\small \begin{equation*}
        \left|\frac{\partial^2 f_i}{\partial x^2}(\bar{x}(t,\o),\o)\right| \leq C_{f_i}^2 ,\ \o \in \supp(\mu), \text{ a.e. } t \in [0,T].
    \end{equation*}}
\end{itemize}

\begin{theorem}{\cite[Theorem 3.3]{Bettiol2019}}
\label{pmp}
Assume that conditions ({\bf H0}) and ({\bf H1}) hold, and let $(\bar{u}, \{\bar{x}(\cdot, \omega) : \omega \in \Omega \} )$ be a $W^{1,1}$-local minimizer for $(P)$. Then, there exists a function $p : [0,T] \times \Omega \rightarrow \mathbb{R}^n$ that is measurable with respect to $I \times \mathcal{B}_\Omega$, such that:
\begin{equation*}
    p(\cdot, \omega) \in W^{1,1}([0,T], \mathbb{R}^n) \quad \text{for all } \omega \in \supp(\mu),
\end{equation*}
{\small \begin{equation}\label{pmp.finite}
\begin{aligned}
    &\int_\Omega p(t,\omega) f(\bar{x}(t,\omega), \bar{u}(t), \omega) \, d\mu(\omega) = \\ &\max_{u \in U(t)} \int_\Omega p(t,\omega) f(\bar{x}(t,\omega), u, \omega) \, d\mu(\omega) \ \text{a.e. } t \in [0,T],
\end{aligned}
\end{equation}}

\begin{equation}\label{pmp:adjoint}
\begin{aligned}
    -\dot{p}(t,\omega) = [\nabla_x f(\bar{x}(t,\omega), \bar{u}(t), \omega)]^\top p(t,\omega), \\
    \text{a.e. } t \in [0,T], \text{ for all } \omega \in \mathrm{supp}(\mu),
    \end{aligned}
\end{equation}

\begin{equation*}
    -p(T,\omega) = \nabla_x g(\bar{x}(T,\omega), \omega), \text{ for all } \omega \in \mathrm{supp}(\mu).
\end{equation*}
\end{theorem}
Building on the previous discussions, we confirm that the theorem stated above holds in our framework. Condition \eqref{pmp.finite} will be used to leverage the result and characterize the optimal control.

\section{control characterization}\label{control.char}

In this section, we establish results for the system $(P)$ by leveraging the control-affine framework. We analyze the scalar control case under commutative system dynamics and provide a characterization of singular control for the given system. We begin by examining the regularity properties of the solution to the adjoint system \eqref{pmp:adjoint}. The proof of the following lemma is based on a change of variables and an application of Grönwall's inequality, so we omit the details.
\begin{lemma}\label{lemma_xp}
For a given optimal process $(\bar u, \{\bar x(.,\o) : \o \in \O \})$, there exists a positive constant  $C_p$  such that, for every $\o\in\O$ the following holds true 
\begin{enumerate}[i)]
    \item $|p(t,\o)| \leq C_p$ for all $t\in[0,T]$, 
    \item the map $\o\mapsto p(\cdot,\o)$ is continuous. 
\end{enumerate}
\end{lemma}

\begin{theorem}
[Characterization of the optimal control]
\label{sing_control_scalar}
Let Assumptions ({\bf H0}) and ({\bf H1}) hold. Then, there exists a $W^ {1,1}$-minimizer $(\bar u, \{\bar{x}(.,\o) : \o \in \O \})$ with $p$ as in Theorem \ref{pmp}.
Additionally, the optimal control $\bar u$ satisfies the conditions
\begin{equation}
\label{controls}
  \bar u(t) =
    \begin{cases}
      u_{\max} & \text{if } \Psi(t) > 0, \\
      u_{\min} & \text{if } \Psi(t) < 0, \\
      \mathrm{ singular } & \text{if } \Psi(t) = 0,
    \end{cases}       
\end{equation}
where $\Psi$ is the {\em switching function} that is given by
\begin{equation}
\label{switching_func}
\Psi(t) = \int_\O p(t,\o)f_1(\bar{x}(t,\omega),\o) d\mu(\o).
\end{equation}
Moreover, over singular arcs the control $\bar u$ satisfies
\begin{equation*}
  \int_\Omega p[f_0,[f_0,f_1]] d\mu(\omega) + \bar u \int_\Omega p[f_1,[f_0,f_1]] d\mu(\omega) =0. 
\end{equation*}
Hence, whenever $\int_\Omega p[f_1,[f_0,f_1]] d\mu(\omega)\neq 0$ the singular arcs verifies
\begin{equation}\label{control_formula}
    \bar u(t) =   -  \frac{\int_\Omega p[f_0,[f_0,f_1]] d\mu(\omega)}{\int_\Omega p[f_1,[f_0,f_1]] d\mu(\omega)} .
\end{equation}
\end{theorem}
\begin{proof}
From the maximum condition of the Pontryagin Maximum Principle given in Theorem \ref{pmp} applied to our control-affine problem  in the scalar control case, we get that, for a.e. $t \in [0,T],$
\begin{equation*}
    \bar{u}(t) \in \argmax_{u \in [u_{\min},u_{\max}]} \int_\O p(t,\o) \big(f_0 +  u f_1 \big) d\mu(\o),
\end{equation*}
which gives
$
    \bar{u}(t) \in \argmax_{u \in [u_{\min},u_{\max}]}  u\Psi(t).
$
Therefore, if $\Psi(t) \neq 0,$ one gets one of the two situations on the first two rows of \eqref{controls}. When $\Psi(t)=0$, we can differentiate w.r.t. the time variable to get an expression on the control $
\bar u$. 

Omitting the functions dependencies for convenience and using that $\Psi = \int_\O \varphi d\mu(\o)$, with $\varphi = pf_1$ it is clear that
\begin{equation}
\label{dphi_1}
    \left|  \frac{d\varphi}{dt} \right| = \left|p(f_1' f_0 - f_0' f_1) \right|,
\end{equation}
By assumption {\bf (H1)}-(iv) and from \eqref{dphi_1} we get
\begin{equation*}
    \left|  \frac{d\varphi}{dt} \right| \leq C^1|p|\left|f_0 - f_1\right| \text{ for } t\in[0,T] ,
\end{equation*}
where $C^1:=\max\big\{C^1_{f_0},C^1_{f_1}\big\}$ and with 
Lemma 
\ref{lemma_xp}, we obtain
$
 \left|  \frac{d\varphi}{dt} \right| \leq K^1 \text{ for } t\in[0,T]
$
for some constant $K^1$. Therefore passing the derivative under the integral sign is allowed and results in
\begin{equation*}
    \dot{\Psi} = \int_\Omega p[f_1,f_0]d\mu(\omega)=0.
\end{equation*}
Defining $\eta = p[f_1,f_0]$, one can see that
$$
\left| \frac{d\eta}{dt}\right| = -up[f_1,[f_0,f_1]] + p[f_0,[f_0,f_1]]
$$
Condition {\bf (H1)}-(iv) together with Lemma \ref{lemma_xp} provide that there is some constant $K^2$ such that 
$
\left| \frac{d\eta}{dt}\right| \leq K^2.
$
Once again the derivation under the integral sign is well defined leading to
\begin{equation*}
    \ddot{\Psi} = \int_\Omega p[f_1,[f_0,f_1]]u + p[f_0,[f_0,f_1]]d\mu(\omega) = 0.
\end{equation*}
Then, formula \eqref{control_formula} holds true,
provided that $p[f_1,[f_0,f_1]] \neq 0.$
\end{proof}

\section{Numerical examples}\label{NUM_EXMP}
\subsection{Numerical scheme}
To numerically solve the problem $(P)$, we apply the {\em sample average approximation method.} More specifically, at each iteration $k$, a finite set of independent and identically $\mu$-distributed samples, $\O_k = \{\o_i^k\}_{i=1}^k$, is selected from the parameter space $\Omega$. Then, the following {\it standard} finite-dimensional optimal control problem, which we will call $(P)_{k}$  is considered:
\begin{equation*}
  \min_{u} J_{k}[u(\cdot)], \ \text{where }J_{k}[u(\cdot)]:= \frac{1}{k}\sum_{i=1}^{k}g(x(T, \o_i^k),\o_i^k), \\
\end{equation*}
subject to the constraints
\begin{equation*}
 \begin{cases}
\dot{x}(t,\o_{i}^k) = f_0(x(t, \o_{i}^k),\o_{i}^k) + f_1(x(t, \o_{i}^k),\o_{i}^k)u(t), \\
x(0,\o_{i}^k)=x_0, \quad i=1,\dots,k.
\end{cases}
\end{equation*}
 In this section, we assume that the set of admissible controls is defined by $ \hat{\mathcal{U}}_{\rm ad} := \left\{ u \in  \mathcal{U}_{\rm ad} \cap L^2(0,T;\mathbb{R}^m) \right\}. $ 
Note that $\hat{\mathcal{U}}_{\rm ad} $ is a complete and separable metric space with respect to the $ L^2 $-topology. Therefore, we can apply the numerical approach developed in \cite{Phelps2016}, where the authors use 
the Monte Carlo method. This numerical method estimates the objective functional $ J $ defined in problem $ (P) $ by utilizing the sample mean $ J_k $ defined in $ (P)_k $. The convergence properties of the problems $ (P)_k $ are examined in \cite{Phelps2016} through the concept of {\em epi-convergence} of the objective functionals. This type of convergence provides a natural framework for analyzing the approximation of optimization problems, as it allows for the discussion of the convergence of both their minimizers and infima.

\subsection{Example: Sterile Insect Technique }

We consider a model for the implementation of the {\it sterile insect technique} (SIT), which is a biological control tool that aims at reducing the size of an insect population by invading the field with radiation-steriled insects. The model in \eqref{sit_model} below is a modification of the ones in \cite{AronnaDumont} and \cite{BidiCoronHayat}. 

The system consists of four state variables: $A$ representing the insect population in aquatic stage, $F$ adult females, $M$ adult males, and $M_s$ adult sterile individuals. The control $u$ represents the rate at which sterile males are released into the population. The problem of optimizing the sterile insect release, which we will refer to as $(P)_{SIT},$ is formulated as follows:
\begin{equation*}
\begin{aligned}
 &\min_{u}J_{\rm SIT}[u(\cdot)], \text{ where } \hspace{5cm}
 \end{aligned}
 \end{equation*}
 {\small \begin{equation*}
J_{\rm SIT}[u(\cdot)]:=\int_\O 
\left( c_1\int_{0}^T u(t) dt + c_2 (F(T) + M(T)) \right) d\mu(\omega),
\end{equation*}}
subject to 
\begin{equation}\label{sit_model}
\begin{cases}
 &\dot A(t,\omega)=\frac{\alpha M(t,\omega)F(t,\omega)}{M(t,\omega) + \gamma M_s(t,\omega)} \left(1-\frac{A(t,\omega)}{k} \right)
 \\
 &\hspace{2.6cm}-\mu_A A(t,\omega) - \nu A(t,\omega),
 \\
&\dot F(t,\omega)=r \nu A(t,\omega) - \mu_F F(t,\omega) ,
\\
&\dot M(t,\omega)=(1-r)\nu A(t,\omega)-\mu_M M(t,\omega),
\\
&\dot M_s(t,\omega)=u(t) - \mu_{s}M_s(t,\omega),
    \end{cases}
\end{equation}
with initial conditions $$(A_0,F_0,M_0,M_{{s}_0}) = (19941, 14956, 12962,0),$$ and where $0\leq u(t) \leq 2\times 10^5$ a.e. $t\in [0,T].$ We consider fixed parameters $T = 90$, $\alpha = 6.66$, $\gamma=0.91$, $r = 0.5$, $k=20000$, $c_1=0.15$, $c_2=200$ and stochastic parameters $\nu \sim \text{Uniform}(0.09,0.11)$, $\mu_A\sim \text{Uniform}(0.009,0.01)$, $\mu_F\sim\text{Uniform}(0.0625,0.0714)$, $\mu_M\sim \text{Uniform}(0.0714,0.083)$, $\mu_S\sim \text{Uniform}(0.111,0.125),$ where $\text{Uniform}(a,b)$ is the uniform distribution over the interval $[a,b].$

After introducing a new state variable $\dot z(t,\omega) = c_1 u(t)$, omitting the arguments, the dynamics is then given by 
\begin{equation*}
\Tilde{f} = \begin{pmatrix}
\frac{\alpha MF}{M + \gamma M_s} \left(1-\frac{A}{k} \right)-\mu_A A - \nu A\\
    r \nu A - \mu_F F\\
    (1-r)\nu A-\mu_M M\\
    -\mu_sM_s\\
    0
\end{pmatrix}
+ u \begin{pmatrix}
    0 \\
    0 \\
    0 \\
    1 \\
    c_1
\end{pmatrix}
\end{equation*}
and denoting the system's state  as $\Tilde{x}=\left(A, F, M, M_s, z \right)^T$ we transform the original problem $(P)_{SIT}$ into a Mayer problem:
\begin{equation}\label{fishing_mayer}
   \begin{aligned}
    & \min_{u} J_M[u], \text{ where }
    \\
    &J_M[u]:=\int_\O  z(T,\o) +c_2 (F(T) + M(T))\, d\mu(\o) \\
    &\text{subject to}\\
   &\begin{cases}
       \dot {\Tilde x}(t,\o) =
       \Tilde{f}(\Tilde{x}(t, \omega), u(t)) \\
        0\leq u(t) \leq 2\times 10^5\quad \text{ a.e. on }[0,T]\\
        \Tilde{x}(0,\o) = (19941, 14956, 12962,0,0)
    \end{cases}
    \end{aligned}
\end{equation}

so the singular arc, if it exists, is given by \eqref{control_formula}.
We approximate the solution of \eqref{fishing_mayer} by solving the discretized version given in  $(P)_k$ for sufficiently large values of $k$. In order to solve  $(P)_k$ we use the GEKKO library \cite{beal2018gekko}, which employs direct discretization of the optimal control problem. For $k = 26$, we obtained the results given in Figures \ref{fig:fishing cost} and \ref{fig:control fishing}, with associated errors $\displaystyle\frac{|J^{24} - J^{25}|}{|J^{24}|}=4.45\times 10^{-4}$ and $\frac{\|u^{24} - u^{25}\| }{\|u^{24}\|}=4.93\times10^{-3}$.
\begin{figure}[H]
    \centering
    \includegraphics[width=0.32\textwidth]{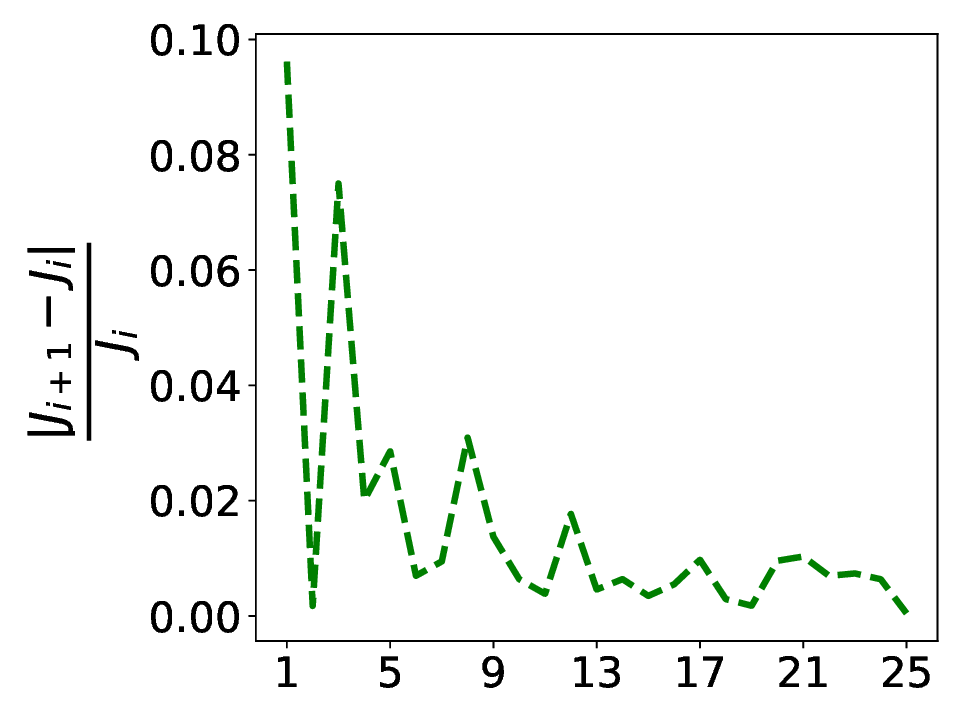}
    \caption{Relative distances of the cost functional between successive iterations.}
    \label{fig:fishing cost}
\end{figure}

\begin{figure}[H]
    \centering
    \includegraphics[width=0.32\textwidth]{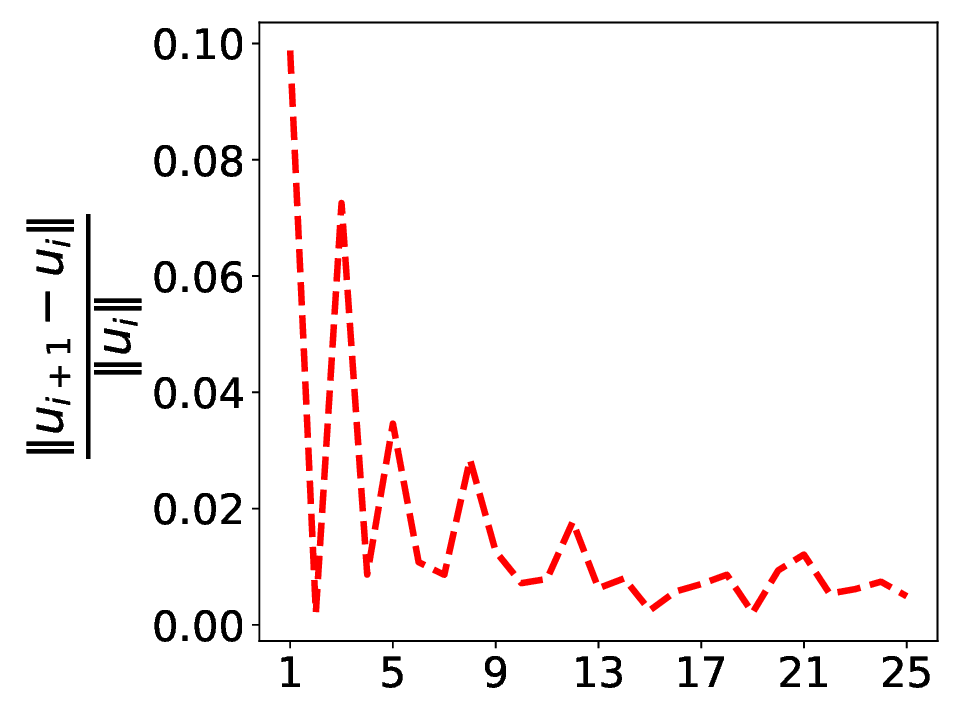}
    \caption{Relative distances of the control between successive iterations.}
    \label{fig:control fishing}
\end{figure}

\begin{figure}[htbp]
\begin{center}
\includegraphics[width=.8\linewidth]{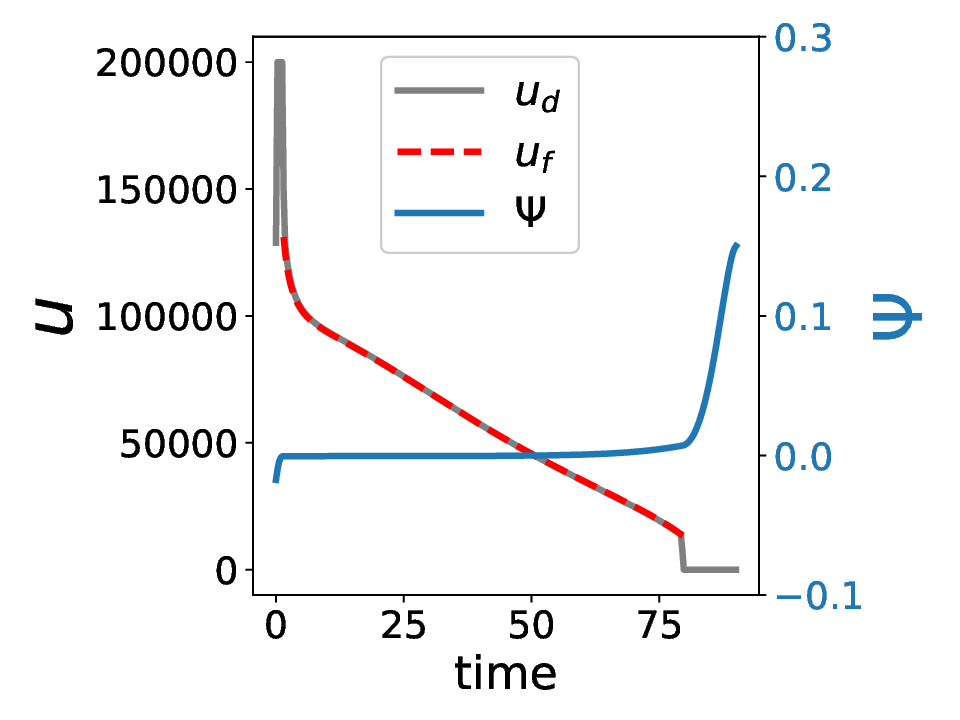}
\caption{In dashed red, the optimal control calculated with the formula \eqref{control_formula}. In grey, the optimal control calculated with the GEKKO library. In blue the switching function $\Psi$ given by \eqref{switching_func}.}\label{fig:control_fishing}
\end{center}
\end{figure}
\section{Conclusions}\label{conclusions}
In this work, we examined the existence of optimal controls for problem $(P)$ and applied the Pontryagin Maximum Principle from \cite{Bettiol2019} to characterize singular arcs for scalar optimal controls in a feedback form. As an application, we introduced a model for the sterile insect technique and validated our approach through numerical implementation. To solve problem $(P)$ computationally, we employed the sample average approximation method. The numerical results confirm the accuracy of our method, as the control trajectories obtained closely align with those generated using Python’s GEKKO optimization library (Figure \ref{fig:control_fishing}), illustrating the validity of our theoretical findings. Additionally, the graphs showing the cost relative distance (Figure \ref{fig:fishing cost}) and control relative distance (Figure \ref{fig:control fishing}) between successive iterations indicate convergence, as both trends approach zero. These results highlight the effectiveness and stability of the proposed framework, suggesting its potential for broader applications in uncertain optimal control problems. 
\bibliographystyle{plain}
\bibliography{ref2}
\end{document}